\documentclass[a4paper,12pt,reqno]{amsart}
\usepackage[utf8]{inputenc}
\usepackage{amsmath}
\usepackage{amsthm,amssymb,amsfonts}
\usepackage{graphicx}
\usepackage{hyperref}
\usepackage{url}

\usepackage{geometry}
\usepackage{array,bm}
\usepackage{color}
\usepackage{verbatim} 

\newtheorem{theorem}{Theorem}

\newtheorem{remark}[theorem]{Remark}

\newtheorem{lemma}[theorem]{Lemma}

\begin{document}
\title[Non-radial Galbis via Nicola-Tilli]{Norm bounds for self-adjoint Toeplitz operators with non-radial symbols on the Fock space}
\author{Yi C. Huang} 
\address{Yunnan Key Laboratory of Modern Analytical Mathematics and Applications, Yunnan Normal University, Kunming 650500, People's Republic of China}
\address{Department of Mathematical Sciences, Tsinghua University, Beijing 100084, People's Republic of China}
\address{School of Mathematical Sciences, Nanjing Normal University, Nanjing 210023, People's Republic of China}
\email{Yi.Huang.Analysis@gmail.com}
\urladdr{https://orcid.org/0000-0002-1297-7674}
\author{Jian-Yang Zhang}
\address{School of Mathematical Sciences, Nanjing Normal University, Nanjing 210023, People's Republic of China}
\email{3203927393@qq.com}
\thanks{Research of the authors is partially supported by the National NSF grant of China (no. 11801274), 
	the Visiting Scholar Program from the Department of Mathematical Sciences of Tsinghua University,
	and the Open Project from Yunnan Normal University (no. YNNUMA2403). YCH thanks Shunlong Luo, Congwen Liu and Youjiang Lin for helpful communications.}
\maketitle
\begin{abstract}{In this paper we extend Galbis' elegant norm bounds for self-adjoint Toeplitz operators on the Fock space to bounded and integrable symbols which are non-radial. The main ingredients are a transplantation of the remarkable Nicola-Tilli isoperimetric inequality to the realm of Fock-Toeplitz operator theory and a two-dimensional adaption of Galbis' integration and approximation arguments.}\end{abstract}

\section{Introduction}
Using the Hermite expansion, Galbis obtained in \cite{galbis2022norm} inspiring norm bounds for self-adjoint Toeplitz operators with radial, bounded and integrable symbols on the Fock space defined on the complex plane. See also Grudsky and Vasilevski \cite{grudsky2002toeplitz} for a related result. In this paper we extend Galbis' norm bounds to the non-radial case. The main ingredient is a Gaussian isoperimetric inequality established recently by Nicola and Tilli in \cite{nicola2022faber}, and for readers' convenience we recall it as follows.
\begin{theorem}[Nicola-Tilli, \cite{nicola2022faber}]\label{thm NT}
	For every $f\in{\mathcal F}^2({\mathbb C})$ with $||f||_{{\mathcal{F}}}=1$ and every measurable set $\Omega\subset\mathbb{C}$ of finite measure, there holds
		$$\int_{\Omega}|f(z)|^2e^{-\pi|z|^2}\,dA(z)\le1-e^{-|\Omega|},$$
	where $|\Omega|$ is the Lebesgue measure of $\Omega$.
\end{theorem}

\begin{remark}
For more recent progresses following this remarkable result (in different background geometry or its stability), we mention for example Kulikov \cite{kulikov2022functionals}, Frank \cite{frank2023sharp}, and G\'omez-Guerra-Ramos-Tilli \cite{gomez2024stability} (and the references therein).
\end{remark}

In the statement of Theorem \ref{thm NT}, the Fock (or Bargmann-Fock) space ${\mathcal F}^2({\mathbb C})$ is the Hilbert space consisting of those entire functions $f$ such that
$$
\|f\|_{{\mathcal F}}^2 = \int_{{\mathbb C}}|f(z)|^2 e^{-\pi |z|^2}\ dA(z) < +\infty,
$$ where $dA(z)$ denotes the Lebesgue measure. It is well known that ${\mathcal F}^2({\mathbb C})$ admits a reproducing kernel $K(z,w) = e^{\pi \overline{w}z},$ which means that $$f(w)=\int_{{\mathbb C}}f(z)\overline{K(z,w)}e^{-\pi |z|^2}\ dA(z),\ \ f\in {\mathcal F}^2({\mathbb C}).$$  
 For simplicity we denote $d\lambda(z) = e^{-\pi |z|^2} dA(z),$ so ${\mathcal F}^2({\mathbb C})$ is a closed subspace of $L^2({\mathbb C}, d\lambda).$ The induced orthogonal projection 
$$
\mathbb{P}:L^2({\mathbb C}, d\lambda)\to {\mathcal F}^2({\mathbb C})$$ is the integral operator 
$$
\left(\mathbb{P}f\right)(z) = \int_{{\mathbb C}} f(w)\overline{K(w,z)}\ d\lambda(w).$$ For a measurable and bounded function $\varphi$ on ${\mathbb C}$, the Toeplitz operator associated with symbol $\varphi$ is defined as 
$$
T_\varphi(f)(z) = \mathbb{P}(\varphi f)(z) = \int_{{\mathbb C}} \varphi(w)f(w)\overline{K(w,z)}\ d\lambda(w).$$ 
It is obvious that $$T_\varphi:{\mathcal F}^2({\mathbb C})\to {\mathcal F}^2({\mathbb C})$$ is a bounded operator and 
$$
\|T_\varphi(f)\|_{\mathcal{F}} \leq \|\varphi f\|_{L^2({\mathbb C}, d\lambda)} \leq \|\varphi\|_\infty\cdot \|f\|_{\mathcal{F}}.$$ 
The Toeplitz operator defined by a real valued symbol $\varphi$ is self-adjoint. This can be seen via the following pairing identity

$$
\langle T_\varphi(f), g\rangle = \int_{{\mathbb C}} \varphi(z) f(z)\overline{g(z)}d\lambda(z)$$ for all $f,g\in {\mathcal F}^2({\mathbb C}).$ In this case we have
\begin{equation}\label{obeq}
	\|T_\varphi\| = \sup_{\|f\|_{\mathcal{F}}=1}\left|\langle T_\varphi(f), f\rangle\right| \leq \sup_{\|f\|_{\mathcal{F}}=1}\int_{{\mathbb C}} |\varphi(z)|\cdot |f(z)|^2d\lambda(z).
\end{equation}

In order to estimate $||T_\varphi||$, we only need to estimate
\begin{equation*}
	\int_{{\mathbb C}} |\varphi(z)|\cdot |f(z)|^2d\lambda(z),
\end{equation*}
where $f\in\mathcal{F}^2(\mathbb{C})$ such that $||f||_{\mathcal{F}}=1$.
Using Theorem \ref{thm NT} of Nicola and Tilli, and adapting \cite[Lemma 2]{galbis2022norm} of Galbis, we obtain the following result.
\begin{theorem}\label{th:main} Let $\varphi\in L^1({\mathbb C})\cap L^\infty({\mathbb C})$ be a real-valued symbol. Then 
	\begin{equation}\label{eq norm bound}
		\|T_\varphi\|\leq \|\varphi\|_\infty \left(1 - \exp\left(-\frac{\|\varphi\|_1}{\|\varphi\|_\infty}\right)\right).
	\end{equation}
	
\end{theorem}

\begin{remark}
	Examing the statements in Nicola and Tilli \cite{nicola2022faber}, it is also transparent that the equality in \eqref{eq norm bound} is attained by choosing $\varphi$ as the indicator function of a disc with measure $\|\varphi\|_1$ and the operator norm in \eqref{obeq} is saturated by $f(z)=K(z,w_0)/K(w_0,w_0)$ with $w_0$ being the center of that disc. 
\end{remark}

\section{Proof of Theorem \ref{thm NT}}
First, we prove the following simple but very useful integration lemma which is a two-dimensional adaption of \cite[Lemma 2]{galbis2022norm}.
\begin{lemma}\label{lemma finte e}
	Let $(\Omega_k)_{k=1}^N$ be disjoint sets contained in $\mathbb{C}$ with finite measure and $0\leq \varepsilon_k\leq 1$ for every $1\leq k\leq N.$ Then, for every $f\in {\mathcal F}^2(\mathbb{C})$ such that $||f||_{{\mathcal F}}=1$ we have
	$$
	\sum_{k=1}^N \varepsilon_k\int_{\Omega_k}|f(z)|^2e^{-\pi|z|^2}\ dA(z)\leq 1 - \exp\left(-\sum_{k=1}^N \varepsilon_k |\Omega_k|\right).$$
\end{lemma}
\begin{proof}
	We denote by $n$ the number of indices $k$ such that $0 < \varepsilon_k < 1$ and we proceed by induction on $n.$ For $n = 0$, by Theorem \ref{thm NT} we have
	 $$\sum_{k=1}^N \int_{\Omega_k}|f(z)|^2e^{-\pi|z|^2}\ dA(z)\leq 1 - \exp\left(-\sum_{k=1}^N|\Omega_k|\right).$$

	 Let us now assume $n = 1.$ Let $1\leq j\leq N$ be the coordinate with the property that $0 < \varepsilon_j <1 $ and  define function $\psi$ by:
	\begin{align*}
		\psi(\varepsilon):= &\sum_{k\neq j}\int_{\Omega_k}|f(z)|^2e^{-\pi|z|^2}\ dA(z) + \varepsilon \int_{\Omega_j}|f(z)|^2e^{-\pi|z|^2}\ dA(z) \\
		&+ \exp\left(-\sum_{k\neq j}|\Omega_k| - \varepsilon |\Omega_j|\right),\,\, \qquad0 \leq \varepsilon \leq 1.
	\end{align*}
	    We aim to prove $\psi(\varepsilon)\le1$. In fact, $\psi(0) \leq 1$ and $\psi(1)\leq 1$ follow from Theorem \ref{thm NT}. Moreover, the critical point $\varepsilon_0$ of $\psi$ satisfies 
	
	\[\int_{\Omega_j}|f(z)|^2e^{-\pi|z|^2}\ dA(z) = |\Omega_j|\exp\left(-\sum_{k\neq j}|\Omega_k| - \varepsilon_0 |\Omega_j|\right).\] Hence
	
	\begin{align*}
		\psi(\varepsilon_0) & = \sum_{k\neq j}\int_{\Omega_k}|f(z)|^2e^{-\pi|z|^2}\ dA(z) + \varepsilon_0|\Omega_j|\exp\left(-\sum_{k\neq j}|\Omega_k| - \varepsilon_0 |\Omega_j|\right) \\ & \\ & \qquad+ \exp\left(-\sum_{k\neq j}|\Omega_k| - \varepsilon_0 |\Omega_j|\right)\\ & \\ & = \sum_{k\neq j}\int_{\Omega_k}|f(z)|^2e^{-\pi|z|^2}\ dA(z) + \left(1 + \varepsilon_0|\Omega_j|\right)\exp\left(-\sum_{k\neq j}|\Omega_k| - \varepsilon_0 |\Omega_j|\right).
	\end{align*}
	Since 
	$$
	1 + \varepsilon_0|\Omega_j| \leq \exp\left(\varepsilon_0 |\Omega_j|\right)$$ we conclude 
	$$
	\psi(\varepsilon_0) \leq \sum_{k\neq j}\int_{\Omega_k}|f(z)|^2e^{-\pi|z|^2}\ dA(z) + \exp\left(-\sum_{k\neq j}|\Omega_k|\right) \leq 1.$$ 
	\par\medskip
	Let us assume that the lemma holds for $n = \ell$ ($0\leq \ell < N$). Then we consider $n = \ell +1$ and the function $\psi:[0,1]^{\ell+1}\to {\mathbb R}$ defined by 
	\begin{align*}
	\psi({\bm\varepsilon}):= &\sum_{k=1}^{\ell+1}\varepsilon_k \int_{\Omega_k}|f(z)|^2e^{-\pi|z|^2}\ dA(z) + \sum_j\int_{\Omega_j}|f(z)|^2e^{-\pi|z|^2}\ dA(z)\\& \qquad+ \exp\left(-\sum_k\varepsilon_k |\Omega_k| - \sum_j |\Omega_j|\right)
	\end{align*}
	 for ${\bm \varepsilon} = (\varepsilon_1, \ldots, \varepsilon_{\ell+1}).$ The induction hypothesis means that $\psi({\bm \varepsilon}) \leq 1$ whenever ${\bm \varepsilon}$ is in the boundary of $[0,1]^{\ell+1}.$ The lemma is proved after checking $\psi({\bm \varepsilon}_0) \leq 1,$ where ${\bm \varepsilon_0}$ is a critical point of $\psi.$ Proceeding as before,
	$$
	\begin{array}{*2{>{\displaystyle}l}}
		\psi({\bm \varepsilon}_0) & = \left(\sum_{k=1}^{\ell+1}\varepsilon_k |\Omega_k| + 1\right)e^{-\sum_k \varepsilon_k |\Omega_k|} e^{-\sum_j |\Omega_j|} + \sum_j \int_{\Omega_j}|f(z)|^2e^{-\pi|z|^2}\ dA(z) \\ & \\ & \leq \exp\left(-\sum_j |\Omega_j|\right) + \sum_j \int_{\Omega_j}|f(z)|^2e^{-\pi|z|^2}\ dA(z) \leq 1.
	\end{array}	$$
	
By induction, we finish the proof of Lemma \ref{lemma finte e}.\end{proof}
Next we give a proof of Theorem \ref{th:main}.
\begin{proof}
	 After replacing $\varphi$ by $\frac{\varphi}{\|\varphi\|_\infty}$ if necessary we can assume that $\|\varphi\|_\infty = 1.$ According to \eqref{obeq}, we aim to prove that 
	$$
	\int_{\mathbb C}|\varphi(z)|\left|f(z)\right|^2 e^{-\pi|z|^2}\ dA(z) \leq 1 - \exp(-||\varphi||_1)$$ for every entire function $f\in\mathcal{F}^2(\mathbb{C})$ such that $||f||_{\mathcal{F}}= 1.$
	Let us first assume 
	\begin{equation}\label{eq:stepfunction}
		\varphi = \sum_{k=1}^N \varepsilon_k \chi_{\Omega_k},\ \left|\varepsilon_k\right|\leq 1,
	\end{equation} where $(\Omega_k)_{k=1}^N$ are disjoint measurable sets in $\mathbb{C}$ with finite measure. Then, Lemma \ref{lemma finte e} gives 
	$$
	\begin{array}{*2{>{\displaystyle}l}}
		\int_{\mathbb C}|\varphi(z)|\left|f(z)\right|^2 e^{-\pi|z|^2}\ dA(z) & = \sum^{N}_{k=1 }|\varepsilon_k|\int_{\Omega_k}|f(z)|^2e^{-\pi|z|^2}\ dA(z)\\ & \\ & \le1 - \exp\left(-\sum_{k=1}^N |\varepsilon_k| |\Omega_k|\right)\\ & \\ & = 1 - \exp\left(-\|\varphi\|_1\right).\end{array}$$ Theorem \ref{th:main} is proved for $\varphi$ as in (\ref{eq:stepfunction}). Let us now assume that $\|\varphi\|_\infty = 1$ and $\varphi\in L^1({\mathbb C}).$ Since $\varphi\in L^1(\mathbb{C})$, there is a sequence $(\varphi_n)_n$ of simple functions as in (\ref{eq:stepfunction}) such that $$\lim_{n\to\infty}\int_\mathbb{C} |\varphi_n(z)-\varphi(z)|\ dA(z) = 0.$$ 
	By triangle inequality, we have
	\begin{align*}
		\int_{\mathbb C}|\varphi(z)|\left|f(z)\right|^2 e^{-\pi|z|^2}\ dA(z)&\le\int_{\mathbb C}|\varphi_n(z)|\left|f(z)\right|^2 e^{-\pi|z|^2}\ dA(z)\\\notag
		&\qquad+\int_{\mathbb C}|\varphi(z)-\varphi_n(z)|\left|f(z)\right|^2 e^{-\pi|z|^2}\ dA(z)\\\notag
		&\le1 - \exp\left(-\|\varphi_n\|_1\right)\\\notag
		&\qquad+\int_{\mathbb C}|\varphi(z)-\varphi_n(z)|\left|f(z)\right|^2 e^{-\pi|z|^2}\ dA(z).
	\end{align*}
	Because for a function $f\in\mathcal{F}^2$ such that $||f||_{\mathcal{F}}=1$, we have 
	\[|f(z)|^2e^{-\pi|z|^2}\le1,\]
	see for example \cite[Theorem 2.7]{zhu2012analysis} or \cite[Proposition 2.1]{nicola2022faber}. So 
	\[\int_{\mathbb C}|\varphi(z)-\varphi_n(z)|\left|f(z)\right|^2 e^{-\pi|z|^2}\ dA(z)\le\int_\mathbb{C} |\varphi_n(z)-\varphi(z)|\ dA(z).\]
	Hence we obtain 
	\[\int_{\mathbb C}|\varphi(z)|\left|f(z)\right|^2 e^{-\pi|z|^2}\ dA(z)\le1 - \exp\left(-\|\varphi_n\|_1\right)+\int_\mathbb{C} |\varphi_n(z)-\varphi(z)|\ dA(z).\]
	Letting $n\rightarrow\infty$, we have
	\[\int_{\mathbb C}|\varphi(z)|\left|f(z)\right|^2 e^{-\pi|z|^2}\ dA(z)\le1 - \exp\left(-\|\varphi\|_1\right).\]
	According to \eqref{obeq}, this implies
	$$
	\|T_\varphi\|\leq 1 - \exp\left(-\|\varphi\|_1\right).$$
	Thus the proof of Theorem \ref{th:main} is finished.
	\end{proof}
	
	\bigskip
	
	\section*{\textbf{Compliance with ethical standards}}
	
	\bigskip
	
	\textbf{Conflict of interest} The authors have no known competing financial interests
	or personal relationships that could have appeared to influence this reported work.
	
	\bigskip
	
	\textbf{Availability of data and material} Not applicable.
	
	\bigskip

\bibliographystyle{abbrv}
\bibliography{toep}

\end{document}